\documentclass{book}
\usepackage{tmj}
\usepackage[all,ps,arc]{xy}
\usepackage{amsmath,amssymb,amsfonts}
\usepackage{graphicx}

\begin{document}

\title{On derivations and semiprime ideal of rings} {On derivations and semiprime ideal of rings}

\author{G.~S.~Sandhu, N.~Rehman} {Gurninder Singh Sandhu${}^1$ and Nadeem Ur Rehman${}^2$}

\address{${}^1$Department of Mathematics, Patel Memorial National College, Rajpura 140401, India\\
${}^2$Department of Mathematics, Aligarh Muslim University,  Aligarh, 202002, India.}

\email{gurninder\_rs@pbi.ac.in${}^1$, nu.rehman.mm@amu.ac.in${}^2$}

\received{???} \revised{} \accepted{???}
\volumenumber{???} \volumeyear{???}

\setcounter{page}{1}

\numberwithin{equation}{section}

\maketitle

\begin{abstract}
Let $R$ be an associative ring with a nonzero ideal $I$ and a semiprime ideal $T$ such that $T\subsetneq I.$ Let $K$ be a nonempty subset of $R$ and $d:R\to R$ be a derivation of $R$, if $[d(x),x]\in T$ for all $x\in K,$ then $d$ is said to be a $T$-commuting derivation on $K.$ We show that if some specific $T$-valued differential identities are imposed on $I$, then $d$ is $T$-commuting. Moreover, we provide semiprime ideal variant of some known results on derivations.
\end{abstract}

\metadata{16W25}{16N60, 16W10}{Derivation, semiprime ideal, commuting map}

\section{Introduction and motivation}
Let $R$ be an associative ring with center $Z(R).$ Recall, a nonempty subset $T$ of $R$ is said to be a prime (resp. semiprime) ideal if $aRb\subseteq T$ (resp. $aRa\subseteq T$) implies $a\in T$ or $b\in T$ (resp. $a\in T$), where $a,b\in R.$ A ring $R$ is called a prime ring (resp. semiprime) if $(0)$ is the prime (resp. semiprime) ideal of $R$. To denote commutator of any $x,y\in R,$ we write $[x,y]=xy-yx$. The commutator identities: $[x,yz]=[x,y]z+y[x,z],$ $[xy,z]=x[y,z]+[x,z]y$ shall be used frequently in the sequel. An additive function $\delta:R\to R$ is called a derivation of $R$ if $\delta(xy)=\delta(x)y+x\delta(y)$ for all $x,y\in R;$ a known example of derivation is the mapping $x\mapsto [x,a],$ where $a\in R$ is a fixed element, called inner derivation induced by the element $a$. Let $K$ be a nonempty subset of $R$ and $f:R\to R$ be a mapping. Then it is known that $f$ is commuting (resp. centralizing) on $K$ if $[f(x),x]=0$ (resp. $[f(x),x]\in Z(R)$) for all $x\in K.$ Moreover, if $I$ is a nonzero ideal of $R$ and $T$ is a semiprime ideal of $R$ such that $T\subsetneq I$ (i.e., the ideal structure is also defined on $I/T$), then it is straightforward to see that $aIa\subseteq T$ implies $a\in T,$ where $a\in I;$ this fact shall be very useful for us in the subsequent section.   
\par In 1990, Vukman \cite{Vukman1990} demonstrated that if $d$ is a derivation of a prime ring $R$ with char$(R)\neq 2$ such that the mappings $x\mapsto [d(x),x]$ is a commuting on $R,$ then $R$ is commutative. This result clearly refines Posner's second theorem \cite{Posner57}. Following Posner's seminal work, significant interest has emerged in the study of commuting and centralizing derivations in prime and semiprime rings (see, for instance, \cite{Almahdi2020}, \cite{Bresar2004}, \cite{Oukhtite11}, \cite{Mathiu}). The investigation of derivations in prime rings satisfying specific identities has become a key area of research, particularly in exploring the commutative structure of the underlying ring or characterizing such derivations.

Nowadays, a huge number of results are available in various generalized forms of derivation in prime and semiprime rings. Such works can be extended by taking a broader class of rings than prime rings, i.e., rings with a prime ideal (not necessarily having (0) as prime ideal). It is evident that if one considers a derivation of an arbitrary ring with a prime ideal to establish a commutativity theorem for the ring induced by that prime ideal, then a similar conclusion can be readily derived for prime rings. Perhaps having this idea in mind, Alhamdi et al. recently presented a generalization of Posner's second theorem. They proved that if $P$ is a prime ideal of an arbitrary ring $R$ and $d:R\to R$ is a derivation satisfying $[[d(x),x],y]\in P$ for all $x,y\in R,$ then $d$ maps $R$ into $P$ or $R/P$ is commutative. Let us refer to it as the \emph{$P$-variant} of Posner's second theorem. On the other hand while addressing derivations of prime and semiprime rings one always have a strong tool in hand is \emph{the standard theory of differential and polynomial identities}, which is introduced by Kharchenko  and extensively expanded by algebraists such as Beidar, Mikhlev, Martindale, Jacobson, Chuang, Lee, Bre$\check{s}$ar. But in the present generalization technique of the previous studies, one has this limitation to not to have the access to the PI-theory (i.e. the results which are precisely true for the mappings defined on the class of prime or semiprime rings only). For instance, to be more specific in this vein we now mention some crucial results using this tool: 
\begin{itemize}
	\item Bre$\check{s}$ar \cite{Matej1990} proved that if $R$ is a prime ring, $f,g:R\to R$ are functions such that $f(x)yg(z)=g(x)yf(z)$ for all $x,y,z\in R,$ with $f\neq 0,$ then there exists $\lambda\in C$ (the extended centroid of $R$) such that $g=\lambda f.$
	\item In 1993, Lanski \cite{Lanski1993} proved a generalized version of the theorem of Vukman. The author studied the higher commutator $[d(x),x]_{n}=[\cdots[[d(x),x],x]\cdots]=0$ involving derivation $d$ of a prime ring by using the theory of differential identities.  
\end{itemize}  
Obviously, it is not trivial to claim the $P$-variant of these results. In addition to this, as given in \cite{Krempa1984}, Herstein asked the question that \emph{if $R$ is a semiprime ring and $d:R\to R$ is a derivation. Does every minimal prime ideal $P$ of $R$ invariant under $d$?} This problem was settled by Chuang and Lee \cite{Chuang1991} by refuting it. If this problem had had a positive answer, then obviously the $P$-variants of results on prime rings would be rendered trivial. Therefore, investigation of $P$-variants of previous results is non-trivial and can be quit challenging. In this study, we prove $T$-variant ($T$ is a semiprime ideal of the ring $R$) of some well-established results on derivations in prime and semiprime rings. 

\section{Main Results}
We shall adopt a different technique than that employed by Vukman \cite{Vukman1990} and obtain a more general conclusion with a shorter proof.

\begin{theorem}\label{thm-01}
	Let $R$ be a ring with a nonzero ideal $I$ and a semiprime ideal $T$ such that $T\subsetneq I.$ If char$(R/T)\neq 2$ and $[[d(x),x],x]\in T$ for all $x\in I,$ then $d$ is $T$-commuting on $I$.
	\begin{proof}
		Let us suppose that
		\begin{equation}\label{eq-1}
			[[d(x),x],x]\in T 
		\end{equation}
		for all $x\in I.$ Substituting $x+y$ instead of $x$ in (\ref{eq-1}), we have
		\begin{equation}\label{eq-2}
			[[d(x),x],y]+[[d(x),y],y]+[[d(x),y],x]+[[d(y),x],y]+[[d(y),x],x]+[[d(y),y],x]\in T
		\end{equation}
		for all $x,y\in I.$ Using $-x$ instead of $x$ in (\ref{eq-2}) in order to find
		\begin{equation}\label{eq-2.1}
			[[d(x),x],y]-[[d(x),y],y]+[[d(x),y],x]-[[d(y),x],y]+[[d(y),x],x]-[[d(y),y],x]\in T
		\end{equation}
		for all $x,y\in I.$ Jointly considering (\ref{eq-2}) and (\ref{eq-2.1}), we get
		\[
		2([[d(x),x],y]+[[d(x),y],x]+[[d(y),x],x])\in T
		\]
		for all $x,y\in I.$ Using char$(R/T)\neq 2,$ we obtain
		\begin{equation}\label{eq-3}
			[[d(x),x],y]+[[d(x),y],x]+[[d(y),x],x]\in T 
		\end{equation}
		for all $x,y\in I.$ Putting $yz$ for $y$ in (\ref{eq-3}), one may see that
		\[
		\begin{split}
			[[d(x),x],y]z+y[[d(x),x],z]+[[d(x),y],x]z+[d(x),y][z,x]+y[[d(x),z],x]
			\\+[y,x][d(x),z]+[[d(y),x],x]z+[d(y),x][z,x]+d(y)[[z,x],x]+[d(y),x][z,x]
			\\+y[[d(z),x],x]+[y,x][d(z),x]+[[y,x],x]d(z)+[y,x][d(z),x]\in T
		\end{split}
		\]
		for all $x,y,z\in I.$
		Using (\ref{eq-3}), we find
		\begin{equation}\label{eq-4}
			\begin{split}
				[d(x),y][z,x]+[y,x][d(x),z]+2[d(y),x][z,x]+d(y)[[z,x],x]+2[y,x][d(z),x]
				\\+[[y,x],x]d(z)\in T
			\end{split}
		\end{equation}
		for all $x,y,z\in I.$ Replacing $z$ by $zx$ in the equation (\ref{eq-4}) and using it, we find
		\begin{equation}\label{eq-5}
			3[y,x]z[d(x),x]+2[y,x][z,x]d(x)+[[y,x],x]zd(x)\in T 
		\end{equation}
		for all $x,y,z\in I$. Taking $zx$ instead of $z$ in (\ref{eq-5}), we get
		\begin{equation}\label{eq-6}
			3[y,x]zx[d(x),x]+2[y,x][z,x]xd(x)+[[y,x],x]zxd(x)\in T 
		\end{equation}
		for all $x,y,z\in I$. Equation (\ref{eq-5}) also gives us
		\begin{equation}\label{eq-7}
			3[y,x]z[d(x),x]x+2[y,x][z,x]d(x)x+[[y,x],x]zd(x)x\in T  
		\end{equation}
		for all $x,y,z\in I$. Considering together Eq. (\ref{eq-6}) and (\ref{eq-7}) in order to obtain 
		\begin{equation}\label{eq-8}
			2[y,x][z,x][d(x),x]+[[y,x],x]z[d(x),x]\in T 
		\end{equation}
		for all $x,y,z\in I$. Replacing $y$ by $d(x)y$ in (\ref{eq-8}) and using it along with our initial hypothesis, we find
		\begin{equation}\label{eq-8.1}
			2[d(x),x]y[z,x][d(x),x]+2[d(x),x][y,x]z[d(x),x]\in T
		\end{equation}
		for all $x,y,z\in I.$ Replacing $z$ by $z[d(x),x]t$ in (\ref{eq-8.1}), we get
		\[
		\begin{split}
			2[d(x),x]y[z,x][d(x),x]t[d(x),x]+2[d(x),x]yz[d(x),x][t,x][d(x),x]
			\\+2[d(x),x][y,x]z[d(x),x]t[d(x),x]\in T
		\end{split}
		\]
		for all $x,y,z\in I.$ In view of (\ref{eq-8.1}) it forces that 
		$$
		2[d(x),x]yz[d(x),x][t,x][d(x),x]\in T
		$$
		for all $x,y,z\in I.$ Hence, we obtain
		\begin{equation}\label{eq-8.2}
			2[d(x),x][t,x][d(x),x]\in T
		\end{equation}
		for all $x,t\in I.$
		Now, let us again look at the equation (\ref{eq-8.1}) and substitute $yd(x)$ for $y$ in it. Then we see that
		\[
		2[d(x),x]yd(x)[z,x][d(x),x]+2[d(x),x]y[d(x),x]z[d(x),x]+2[d(x),x][y,x]d(x)z[d(x),x]\in T
		\]
		for all $x,y,z\in I.$
		It implies that
		\begin{equation}\label{eq-9.1}
			\begin{split}
				2[d(x),x]w[d(x),x]yd(x)[z,x][d(x),x]t[d(x),x]+2[d(x),x]w[d(x),x]y[d(x),x]\\
				z[d(x),x]t[d(x),x]+2[d(x),x]w[d(x),x][y,x]d(x)z[d(x),x]t[d(x),x]\in T
			\end{split}
		\end{equation}
		for all $x,y,z,w,t\in T.$ 
		\par Going back to equation (\ref{eq-4}), to collect more useful relations. Replacing $y$ by $xy$ in (\ref{eq-4}) to observe that
		\begin{equation}\label{eq-10}
			3[d(x),x]y[z,x]+2d(x)[y,x][z,x]+d(x)y[[z,x],x]\in T 
		\end{equation}
		for all $x,y,z\in I$. Taking $z=zd(x)$ in (\ref{eq-10}), we have
		\begin{equation}\label{eq-11}
			3[d(x),x]yz[d(x),x]+2d(x)[y,x]z[d(x),x]+2d(x)y[z,x][d(x),x]\in T 
		\end{equation}
		for all $x,y\in I$. Changing $z$ by $[d(x),x]z$ in (\ref{eq-11}) and using (\ref{eq-8.2}) along with our initial hypothesis, we conclude that
		$$
		3[d(x),x]y[d(x),x]z[d(x),x]+2d(x)[y,x][d(x),x]z[d(x),x]\in T 
		$$
		for all $x,y,z\in I.$ It can be seen as
		\begin{equation}\label{eq-11.1}
			3[d(x),x]z[d(x),x]t[d(x),x]+2d(x)[z,x][d(x),x]t[d(x),x]\in T 
		\end{equation}
		for all $x,y,z\in I.$ In the same way, we obtain the following relation form (\ref{eq-5})
		\begin{equation}\label{eq-12}
			3[d(x),x]w[d(x),x]y[d(x),x]+2[d(x),x]w[d(x),x][y,x]d(x)\in T  
		\end{equation}
		for all $x,y,w\in I$. Using (\ref{eq-11}) and (\ref{eq-12}) in (\ref{eq-9.1}), we see that
		\begin{equation}\label{eq-13}
			-4[d(x),x]w[d(x),x]y[d(x),x]z[d(x),x]t[d(x),x]\in T
		\end{equation}
		for all $x,y,z,w,t\in I.$ Using the char$(R/T)\neq 2$, we find 
		\[
		[d(x),x]w[d(x),x]y[d(x),x]z[d(x),x]t[d(x),x]\in T
		\]
		for all $x,y,z,w,t\in I.$ It forces $[d(x),x]\in T$ for all $x\in I.$
	\end{proof}
\end{theorem}

\begin{corollary}\label{coro-1}
	Let $R$ be semiprime ring with char$(R)\neq 2$ and $I$ be a nonzero ideal of $R.$ If $d:R\to R$ be a derivation such that $x\mapsto [d(x),x]$ is a commuting map on $I$, then $d$ is a commuting map on $I$, and hence maps $R$ into $Z(R).$
\end{corollary}



\begin{corollary}\label{coro-2}\cite[Theorem 1]{Vukman1990}
	Let $R$ be noncommutative prime ring with char$(R)\neq 2.$ If $d:R\to R$ be a derivation such that $x\mapsto [d(x),x]$ is a commuting map on $R$, then $d=0.$
\end{corollary}

Moreover, Almahdi et al. in \cite{Almahdi2020} presented a generalization of the well-know result of Posner, known as Posner's second theorem. Specifically, the author proved that
if $R$ is an arbitrary ring, $P$ a prime ideal of $R$, and $d$ is a derivation of $R$ such that $[[x,d(x)], y]\in P$ for
all $x,y\in R$, then $d(R)\subseteq P$ or $R/P$ is commutative. In this view, it is a natural question that whether this conclusion holds for the identity $[[d(x),x],x]\in P$ for all $x,y\in R.$ It can be observed that the following corollary of our Theorem \ref{thm-01} is an affirmation to this question.

\begin{corollary}\label{coro-3}
	Let $R$ be a ring, $P$ be a prime ideal of $R$ and $d$ be a derivation of $R.$ If char$(R/P)\neq 2$ and $[[d(x),x],x]\in P$ for all $x\in R,$ then $d(R)\subseteq P$ or $R/P$ is commutative.
	\begin{proof}
		Reducing the proof of Theorem \ref{thm-01} to prime ideal $P$ and then using \cite[Lemma 2.1]{Almahdi2020}, we get our conclusion. 
	\end{proof}
\end{corollary}
\begin{theorem}\label{thm-02}
	Let $R$ be a ring with nonzero ideal $I$ and a semiprime ideal $T$ such that $T\subsetneq I.$ If char$(R/T)\neq 2$ and $\overline{[[d(x),x],x]}\in Z(R/T)$ for all $x\in I,$ then $d$ is $T$-commuting on $I$.  
	\begin{proof}
		Let us assume that
		\begin{equation}\label{Eqn-1}
			\overline{[[d(x),x],x]}\in Z(R/T) 
		\end{equation}
		for all $x\in R.$ Linearizing this expression, we get
		\begin{equation}\label{Eqn-2}
			\overline{[[d(x),x],y]+[[d(x),y],x]+[[d(y),x],x]}\in Z(R/T) 
		\end{equation}
		for all $x,y\in R.$ Replacing $y$ by $yx$ in (\ref{Eqn-2}), and using (\ref{Eqn-1}), we see that 
		\[
		\begin{split}
			\overline{[[d(x),x],y]x+y[[d(x),x],x]+[[d(x),y],x]x+2y[[d(x),x],x]+2[y,x][d(x),x]}\\
			\overline{+[[d(y),x],x]x}\in Z(R/T)
		\end{split}
		\] 
		for all $x,y\in R.$ Commuting with $x$ and using (\ref{Eqn-2}), we find
		\[
		5[y,x][[d(x),x],x]+2[[y,x],x][d(x),x]\in T
		\]
		for all $x,y\in I.$ Replacing $y$ by $yx$ in the last expression, we conclude that
		\[
		2[[y,x],x][[d(x),x],x]\in T
		\]
		for all $x,y\in I.$ The assumption on characteristic of $R/T$ facilitates us to obtain 
		\begin{equation}\label{Eqn-3}
			[[y,x],x][[d(x),x],x]\in T
		\end{equation} 
		for all $x,y\in I.$ Substituting $yd(x)$ for $y$ in (\ref{Eqn-3}), and using it to get
		\begin{equation}\label{Eqn-4}
			y[[d(x),x],x]^{2}+2[y,x][d(x),x][[d(x),x],x]\in T
		\end{equation}
		for all $x,y\in I.$ Now considering $d(x)y$ in place of $y$ in (\ref{Eqn-4}), we have 
		$$[d(x),x]y[d(x),x][[d(x),x],x]\in T$$ 
		for all $x,y\in I.$ It yields that
		$$[d(x),x][[d(x),x],x]I[d(x),x][[d(x),x],x]\subseteq T$$
		for all $x\in I.$ It implies that $[d(x),x][[d(x),x],x]\in T$ for all $x\in I.$ Our initial hypothesis forces $[d(x),x]y[[d(x),x],x]\in T$ for all $x\in I.$ Thus, from this one can observe that $[[d(x),x],x]\in T$ for all $x\in I.$ Hence, Theorem \ref{thm-01} completes the proof.
	\end{proof}
\end{theorem}

\begin{corollary}\label{coro-1A}
	Let $R$ be semiprime ring with char$(R)\neq 2$ and $I$ be a nonzero ideal of $R.$ If $d:R\to R$ be a derivation such that $x\mapsto [d(x),x]$ is a centralizing map on $I$, then $d$ is a commuting map on $I$, and hence maps $R$ into $Z(R).$
\end{corollary}

In \cite{Vukman1990}, Vukman proved that if $R$ is a noncommutative prime ring $R$ with characteristic different from 2 and 3, admits a derivation $d$ such that $x\mapsto [d(x),x]$ is centralizing on $R,$ then $d=0.$ In the following corollary it is shown that the assumption of char$(R)\neq 3$ in Vukman's theorem is unnecessary. 

\begin{corollary}\label{coro-2A}
	Let $R$ be noncommutative prime ring with char$(R)\neq 2.$ If $d:R\to R$ be a derivation such that $x\mapsto [d(x),x]$ is a centralizing map on $R$, then $d=0.$
\end{corollary}

In 2005, Cheng \cite{Cheng2005} showed that if $R$ is a prime ring and $d$ is a nonzero derivation of $R$ which satisfies the differential identity $[d(x),x]d(x)=0$ on $R,$ then $R$ is commutative. In this section, we study this identity in a semiprime ideal of an arbitrary ring $R$ and hence give a generalization of Cheng's theorem to semiprime rings. We begin by proving a couple of crucial lemmas. 

\begin{lemma}\label{lem-2.1}
	Let $R$ be a ring with nonzero ideal $I$ and a semiprime ideal $T$ such that $T\subsetneq I.$ If char$(R/T)\neq 2$ and $d:R\to R$ be a derivation such that $[d(x),x]x\in T$ for all $x\in I,$ then $d$ is $T$-commuting on $I$. 
\end{lemma}
\begin{proof}
	Assume that
	\begin{equation}\label{U-1}
		[d(x),x]x\in T
	\end{equation}
	for all $x\in I.$ Linearizing it, we find
	\begin{equation}\label{U-2}
		[d(y),x]x+[d(x),y]x+[d(x),x]y\in T
	\end{equation}
	for all $x,y\in I.$ Replacing $y$ by $yw$ in (\ref{U-2}), we see that
	\[
	[d(y)w+yd(w),x]x+[d(x),yw]x+[d(x),x]yw\in T.
	\]
	for all $x,y,w\in I.$ It gives us 
	\begin{equation}\label{U-3}
		\begin{split}
			[d(y),x]wx+d(y)[w,x]x+y[d(w),x]x+[y,x]d(w)x+y[d(x),w]x
			\\+[d(x),y]wx+[d(x),x]yw\in T
		\end{split}
	\end{equation}
	for all $x,y,w\in I.$ Using (\ref{U-2}) in (\ref{U-3}), we obtain 
	\begin{equation}\label{U-4}
		[d(y),x]wx+d(y)[w,x]x-y[d(x),x]w+[y,x]d(w)x+[d(x),y]wx+[d(x),x]yw\in T
	\end{equation}
	for all $x,y,w\in I.$ Considering $xy$ in the place of $y$ in (\ref{U-4}), to find
	\begin{equation}\label{U-5}
		\begin{split}
			x[d(y),x]wx+d(x)[y,x]wx+[d(x),x]ywx+xd(y)[w,x]x+d(x)y[w,x]x-xy[d(x),x]w\\
			+x[y,x]d(w)x+x[d(x),y]wx+[d(x),x]ywx+[d(x),x]xyw\in T
		\end{split}
	\end{equation}
	for all $x,y,w\in I.$ Combining (\ref{U-4}) and (\ref{U-5}), we get
	\[
	d(x)[y,x]wx+2[d(x),x]ywx+d(x)y[w,x]x-x[d(x),x]yw\in T
	\]
	for all $x,y,w\in I.$ Replacing $w$ by $w[d(x),x]$ in the last expression, and using our hypothesis, we conclude that $x[d(x),x]yw[d(x),x]\in T$ for all $x,y,w\in I.$ In particular, we get $x[d(x),x]Ix[d(x),x]\subseteq T$ for all $x\in I.$ Primeness of $T$ assures us that $x[d(x),x]\in T$ for all $x\in I.$ Along with our hypothesis, we get $[[d(x),x],x]\in T$ for all $x\in I.$ Hence the conclusion follows from Theorem \ref{thm-01}.
\end{proof}
\begin{lemma}\label{lem-2}
	Let $R$ be a ring with nonzero ideal $I$ and a semiprime ideal $T$ such that $P\subsetneq I.$ If char$(R/T)\neq 2$ and $d:R\to R$ be a derivation such that $[d^{2}(x),x]x\in T$ for all $x\in I,$ then $d$ is $T$-commuting on $I$. 
\end{lemma}
\begin{proof}
	Let us assume that $[d^{2}(x),x]x\in T$ for all $x\in I.$ Let us linearize this expression to find
	\begin{equation}\label{L1-1}
		[d^{2}(x),y]x+[d^{2}(y),x]x+[d^{2}(x),x]y\in T
	\end{equation}
	for all $x,y\in I.$ Replacing $y$ by $yx$ in (\ref{L1-1}), we get $[yd^{2}(x)+2d(y)d(x),x]x\in T$ for all $x,y\in I.$ It implies
	\begin{equation}\label{L1-2}
		[y,x]d^{2}(x)x+2[d(y)d(x),x]x\in T
	\end{equation}
	for all $x,y\in I.$ Now changing $y$ to $xy$ in (\ref{L1-2}), we may infer that
	\[
	2([d(x),x]yd(x)+d(x)[yd(x),x])x\in T
	\]
	for all $x,y\in I.$ Using char$(R/T)\neq 2,$ we have
	\begin{equation}\label{L1-3}
		[d(x),x]yd(x)x+d(x)[y,x]d(x)x+d(x)y[d(x),x]x\in T
	\end{equation}
	for all $x,y\in I.$ substituting $yd(x)t$ in the place of $y$ in (\ref{L1-3}) in order to get
	\[
	\begin{split}
		\bigg([d(x),x]yd(x)+d(x)[y,x]d(x)+d(x)y[d(x),x]\bigg)td(x)x+d(x)y\bigg([d(x),x]td(x)
		\\+d(x)[t,x]d(x)+d(x)t[d(x),x]\bigg)x-d(x)y[d(x),x]td(x)x\in T
	\end{split}
	\]
	for all $x,y,t\in I.$ In view of (\ref{L1-3}), it follows that 
	\[
	\begin{split}
		\bigg([d(x),x]yd(x)+d(x)[y,x]d(x)+d(x)y[d(x),x]\bigg)td(x)x
		-d(x)y[d(x),x]td(x)\in T
	\end{split}
	\]
	for all $x,y,t\in I.$ Substituting $xt$ instead of $t$ in the last relation and again utilizing (\ref{L1-3}) to find
	\[
	d(x)y[d(x),x]xtd(x)\in T
	\]
	Thus it yields that
	\[
	([d(x),x]x)y([d(x),x]x)t([d(x),x]x)\in T
	\]
	for all $x,y,t\in I.$ Semiprimeness of $T$ forces that $[d(x),x]x\in T$ for all $x\in I.$ Conclusion follows from Lemma \ref{lem-2.1}.
\end{proof}

\begin{theorem}\label{thm-3}
	Let $R$ be a ring with nonzero ideal $I$ and a semiprime ideal $T$ such that $T\subsetneq I.$ If char$(R/T)\neq 2$ and $d:R\to R$ be a derivation such that $[d(x),x]d(x)\in T$ for all $x\in I,$ then $d$ is $T$-commuting on $I$.
\end{theorem}
\begin{proof}
	Assume that
	\begin{equation}\label{L2-1}
		[d(x),x]d(x)\in T    
	\end{equation}
	for all $x\in I.$ Linearizing this relation, we get
	\begin{equation}\label{L2-4}
		[d(x),x]d(y)+[d(x),y]d(x)+[d(y),x]d(x)\in T    
	\end{equation}
	for all $x,y\in I.$ Taking $yx$ instead of $y$ in (\ref{L2-4}) in order to obtain 
	\begin{equation}\label{L2-5}
		[d(x),x]d(y)x+[d(x),x]yd(x)+[d(x),y]xd(x)+[d(y),x]xd(x)+[y,x]d(x)^{2}\in T    
	\end{equation}
	for all $x,y\in I.$ Right multiplying (\ref{L2-4}) by $x$ and subtracting from (\ref{L2-5}), we get
	\begin{equation}\label{L2-6}
		[d(x),x]yd(x)+[d(x),y][x,d(x)]+[d(y),x][x,d(x)]+[y,x]d(x)^{2}\in T  
	\end{equation}
	for all $x,y\in I.$ Replacing $y$ by $yd(x)$ in (\ref{L2-6}), we obtain
	\begin{equation}\label{L2-7}
		\begin{split}
			[d(x),x]yd(x)^{2}+[d(x),y]d(x)[x,d(x)]+[d(y),x]d(x)[x,d(x)]+d(y)[d(x),x][x,d(x)]\\
			+y[d^{2}(x),x][x,d(x)]+[y,x]d^{2}(x)[x,d(x)]+[y,x]d(x)^{3}\in T  
		\end{split}
	\end{equation}
	for all $x,y\in I.$ Right multiplying (\ref{L2-6}) by $d(x)$ and then subtract it from (\ref{L2-7}), we get
	\begin{equation}\label{L2-8}
		\begin{split}
			[d(x),y]d(x)[x,d(x)]+[d(y),x]d(x)[x,d(x)]+d(y)[d(x),x][x,d(x)]\\
			+y[d^{2}(x),x][x,d(x)]+[y,x]d^{2}(x)[x,d(x)]\in T  
		\end{split}
	\end{equation}
	for all $x,y\in I.$ Replacing $y$ by $xy$ in (\ref{L2-8}), we find
	\begin{equation}\label{L2-9}
		\begin{split}
			[d(x),x]yd(x)[x,d(x)]+d(x)[y,x]d(x)[x,d(x)]+[d(x),x]yd(x)[x,d(x)]+\\
			d(x)y[d(x),x]xd(x)\in T  
		\end{split}
	\end{equation}
	for all $x,y\in I.$ Writing $d(x)y$ for $y$ in (\ref{L2-9}), we may infer that
	\[
	d(x)[d(x),x]yd(x)[d(x),x]\in T
	\]
	for all $x,y\in I.$ In the light of semiprimeness of $T,$ we get 
	\begin{equation}\label{L2-10}
		d(x)[d(x),x]\in T   
	\end{equation}
	for all $x\in I.$ Thus, finally we achieve that
	\begin{equation}\label{L2-11}
		[[d(x),x],d(x)]\in T
	\end{equation}
	for all $x\in I.$ From this equation, we first claim the following:
	\begin{itemize}
		\item [(A)] $[d(x),x]^{2}\in T$ for all $x\in I;$
		\item [(B)] $[[d(x),x],x]d(x)\in T$ for all $x\in I,$
	\end{itemize}
	and then we shall use it to prove our main conclusion.
	\par Now, from (\ref{L2-9}), by using (\ref{L2-10}) we have $d(x)y[d(x),x]xd(x)\in T$ for all $x,y\in I.$ Left multiplying this relation with $[d(x),x]x$ to obtain $[d(x),x]xd(x)y[d(x),x]xd(x)\in T$ for all $x,y\in I.$ 
	In view of semiprimeness of $T,$ it yields 
	\begin{equation}\label{L2-11}
		[d(x),x]xd(x)\in T
	\end{equation}
	for all $x\in I.$ Right multiplying (\ref{L2-1}) by $x,$ and then subtracting it from (\ref{L2-11}), we get $[d(x),x]^{2}\in T$ for all $x\in I,$ which is our claim (A). Again left multiplying (\ref{L2-1}) by $x$ and subtracting it from (\ref{L2-11}), we find $[[d(x),x],x]d(x)\in T$ for all $x\in I,$ which proves (B).
	\par
	Now again we are going back to (\ref{L2-4}) and substituting $yz$ for $y$ in it, to observe
	\[
	\begin{split}
		[d(x),x]d(y)z+[d(x),x]yd(z)+[d(x),y]zd(x)+y[d(x),z]d(x)+[d(y),x]zd(x)\\
		+d(y)[z,x]d(x)+y[d(z),x]d(x)+[y,x]d(z)d(x)\in T
	\end{split}
	\]
	for all $x,y,z\in I.$ It can also be seen as
	\begin{equation}\label{L2-14}
		\begin{split}
			\bigg([d(x),x]d(y)\bigg)z+[d(x),x]yd(z)+[d(x),y]zd(x)+y\bigg([d(x),z]d(x)+[d(z),x]d(x)\bigg)
			\\+[d(y),x]zd(x)+d(y)[z,x]d(x)+[y,x]d(z)d(x)\in T
		\end{split}
	\end{equation}
	for all $x,y,z\in I.$ A repeated application of (\ref{L2-4}) in (\ref{L2-14}) yields
	\begin{equation}\label{L2-15}
		\begin{split}
			-\bigg([d(x),y]+[d(y),x]\bigg)d(x)z+[d(x),x]yd(z)-y[d(x),x]d(z)+\bigg([d(x),y]
			\\+[d(y),x]\bigg)zd(x)+d(y)[z,x]d(x)+[y,x]d(z)d(x)\in T
		\end{split}
	\end{equation}
	for all $x,y,z\in I.$ Therefore, we obtain
	\begin{equation}\label{L2-16}
		[[d(x),x],y]d(z)+\bigg([d(x),y]+[d(y),x]\bigg)[z,d(x)]+d(y)[z,x]d(x)+[y,x]d(z)d(x)\in T
	\end{equation}
	for all $x,y,z\in I.$ Replacing $z$ by $zd(x)$ in (\ref{L2-16}) to assure
	\begin{equation}\label{L2-17}
		[[d(x),x],y]zd^{2}(x)+[y,x]zd^{2}(x)d(x)\in T
	\end{equation}
	for all $x,y,z\in I.$ Placing $xy$ instead of $y$ in (\ref{L2-17}), we see that $[[d(x),x],x]yzd^{2}(x)\in T$ for all $x,y,z\in I.$ It easily implies
	\begin{equation}\label{L2-18}
		[[d(x),x],x]Id^{2}(x)\subseteq T
	\end{equation}
	for all $x,y,z\in I.$ Now for $x=z$ in (\ref{L2-16}), we get
	\begin{equation}\label{L2-19}
		[[d(x),x],y]d(x)+\bigg([d(x),y]+[d(y),x]\bigg)[x,d(x)]+[y,x]d(x)^{2}\in T
	\end{equation} 
	for all $x,y\in I.$ Replacing $y$ by $yx$ in (\ref{L2-19}), to obtain
	\[
	\begin{split}
		[[d(x),x],y]xd(x)+y[[d(x),x],x]d(x)+\bigg([d(x),y]x+y[d(x),x]+[d(y),x]x+y[d(x),x]\\
		+[y,x]d(x)\bigg)[x,d(x)]+[y,x]xd(x)^{2}\in T
	\end{split}
	\]
	for all $x,y\in I.$ Using the fact that $d(x)[d(x),x]\in T$ for all $x\in I$ along with (A) and (B), it follows that
	\begin{equation}\label{L2-20}
		[[d(x),x],y]xd(x)+\bigg([d(x),y]+[d(y),x]\bigg)x[x,d(x)]+[y,x]xd(x)^{2}\in T
	\end{equation}
	for all $x,y\in I.$ Right multiplying (\ref{L2-19}) by $x$ and subtracting from (\ref{L2-20}), we arrive at
	\[
	[[d(x),x],y][x,d(x)]+\bigg([d(x),y]+[d(y),x]\bigg)[x,[x,d(x)]]+[y,x][x,d(x)^{2}]\in T
	\]
	for all $x,y\in I.$ Using $d(x)[x,d(x)]\in T$ and $[x,d(x)]d(x)\in T$ for all $x\in I,$
	It yields
	\begin{equation}\label{L2-21}
		[[d(x),x],y][x,d(x)]+\bigg([d(x),y]+[d(y),x]\bigg)[x,[x,d(x)]]\in T
	\end{equation}
	for all $x,y\in I.$ Right multiplying (\ref{L2-21}) by $td^{2}(x)$ and invoking (\ref{L2-18}) to assure that
	\begin{equation}\label{L2-22}
		[[d(x),x],y][x,d(x)]td^{2}(x)\in T
	\end{equation}
	for all $x,t,y\in I.$ Substituting $wy$ in place of $y$ in (\ref{L2-22}), we have $[[d(x),x],w]y[x,d(x)]td^{2}(x)\in T$ for all $x,y,w,t\in I.$ It is not difficult from this relation to find 
	$$[[d(x),x],w]td^{2}(x)I[[x,d(x)],w]td^{2}(x)\in T$$ for all $x,y,w,t\in I.$ It forces that $[[d(x),x],w]td^{2}(x)\in T$ for all $x,w,t\in I.$ Replacing $t$ by $[d(x),x]t$ in the last relation to yield that
	\[
	[d(x),x]w[d(x),x]td^{2}(x)\in T
	\]
	for all $x,w,t\in I.$ Replacing $w$ by $td^{2}(x)w$ in the last expression, we get
	\[
	[d(x),x]td^{2}(x)I[d(x),x]td^{2}(x)\subseteq T
	\]
	for all $x,t\in I.$ It leads us to
	\begin{equation}\label{L2-23}
		[x,d(x)]td^{2}(x)\in T
	\end{equation}
	for all $x,t\in I.$ Linearizing this relation to find
	\begin{equation}\label{L2-24}
		\bigg([y,d(x)]+[x,d(y)]\bigg)td^{2}(x)+[x,d(x)]td^{2}(y)\in T
	\end{equation}
	for all $x,y,t\in I.$ If we substitute $ud^{2}(x)t$ in place of $t$ in (\ref{L2-24}) and using (\ref{L2-23}), we have
	\begin{equation}\label{L2-25}
		\bigg([y,d(x)]+[x,d(y)]\bigg)ud^{2}(x)td^{2}(x)\in T
	\end{equation}
	for all $x,y,t,u\in I.$ Replacing $y$ by $d(x)y$ in (\ref{L2-25}) to see that
	\[
	([x,d(x)]d(y)+[x,d^{2}(x)]y+d^{2}(x)[y,x])ud^{2}(x)td^{2}(x)\in T
	\]
	for all $x,y,t,u\in I.$ Taking $y=x,$ we get $[x,d^{2}(x)]xud^{2}(x)td^{2}(x)\in T$ for all $x,t,u\in I.$ It implies $$[d^{2}(x),x]xu[d^{2}(x),x]xt[d^{2}(x),x]x\in T$$ for all $x,u,t\in I.$ Thus, we obtain $[d^{2}(x),x]x\in T$ for all $x\in I.$ By Lemma \ref{lem-2}, we have the desired conclusion.
\end{proof}
\begin{theorem}\label{thm-4}
	Let $R$ be a ring with nonzero ideal $I$ and a semiprime ideal $T$ such that $T\subsetneq I.$ If char$(R/T)\neq 2$ and $d:R\to R$ be a derivation such that $d(x)[d(x),x]\in T$ for all $x\in I,$ then $d$ is $T$-commuting on $I$. 
\end{theorem}
\begin{proof}
	As the arguments presented in the above theorem, with slight modification, from our hypothesis i.e. $d(x)[d(x),x]\in T$ for all $x\in I,$ we can obtain $[d(x),x]d(x)\in T$ for all $x\in I.$ Hence the proof of this theorem will also follows in the same way, as we have same situations in our hand.
\end{proof}

As a consequence of \cite{Bell87}, we conclude this article with the following result: 

\begin{corollary}\label{Cor-A}
	Let $R$ be a semiprime ring with characteristic different from 2. If $I$ is a nonzero ideal of $R$ and $d:R\to R$ is a derivation of $R$ such that $[d(x),x]d(x)=0$ (or $d(x)[d(x),x]=0$) for all $x\in I$, then $[d(x),x]=0$ for all $x\in I,$  hence $R$ contains a nonzero central ideal.
\end{corollary}



\begin{thebibliography}{12}

\bibitem{Almahdi2020} 
F. A. A. Almahdi, A. Mamouni and M. Tamekkante, \emph{A generalization of Posner's theorem on derivations in rings}, Indian Journal of Pure and Applied Mathematics, 51 (2020), no. 1, 187--194.

\bibitem{Matej1990} 
M. Bre$\check{s}$ar, \emph{Semiderivations of prime rings}, Proceedings of the American Mathematical Society, 108 (1990), no. 4, 859--860.

\bibitem{Bell87}
H. E. Bell and W. S. Martindale III, \emph{Centralizing mappings of semiprime rings}, Canadian Mathematical Bulletin, 30 (1987), no. 1, 92--101. 

\bibitem{Bresar2004}
M. Bre$\check{s}$ar, \emph{Commuting maps: A Survey}, Taiwanese Journal of Mathematics, 8 (2004), no. 3, 361--397. 

\bibitem{Cheng2005}
H. W. Cheng, \emph{Some results about derivations in prime rings}, Journal of Mathematical Research and Exposition, 25 (2005), no. 4, 625--633.

\bibitem{Chuang1991}
C. L. Chuang and T. K. Lee, \emph{Invariance of minimal prime ideals under derivations}, Proceedings of the American Mathematical Society 113 (1991), no. 3, 613--616.

\bibitem{Lanski1993}
C. Lanski, \emph{An Engel condition with derivation}, Proceedings of the American Mathematical Society, 118 (1993), no. 3, 731--734.

\bibitem{Shahoor}
S. Khan, \emph{On semiprime rings with multiplicative (generalized)-derivations}, Beitr\"age zur Algebra und Geometrie, 57 (2016), 119--128.

\bibitem{Krempa1984}
J. Krempa and J. Matczuk, \emph{On the composition of derivations}, Rendiconti del Circolo Matematico di Palermo Series 2, 33 (1984), no. 2, 441--445.

\bibitem{Oukhtite11}
L. Oukhtite, \emph{Posner’s second theorem for Jordan ideals in rings with involution}, Expositiones Mathematicae, 29 (2011), 415--419.

\bibitem{Mathiu}
M. Mathieu, \emph{Posner’s second theorem deduced from the first}, Proceedings of the American Mathematical Society, 114 (1992), 601--602.

\bibitem{Posner57}
E. C. Posner, \emph{Derivations in prime rings}, Proceedings of the American Mathematical Society, 8 (1957), 1093--1100.

\bibitem{Sandhu20}
G. S. Sandhu and D. K. Camci, \emph{Some results on prime rings with multiplicative derivations}, Turkish Journal of Mathematics, 44 (2020), 1401--1411. 

\bibitem{Vukman1990} 
J. Vukman, \emph{Commuting and centralizing mappings in prime rings}, Proceedings of the American Mathematical Society, 109 (1990), no. 1, 47--52.







\end{thebibliography}
\end{document}